\documentclass[]{interact}
\usepackage[utf8x]{inputenc}
\usepackage{epstopdf}
\usepackage[caption=false]{subfig}

\usepackage{xcolor}
\usepackage[longnamesfirst,sort]{natbib}
\bibpunct[, ]{(}{)}{;}{a}{,}{,}

\theoremstyle{plain}
\newtheorem{theorem}{Theorem}[section]

\theoremstyle{definition}
\newtheorem{definition}[theorem]{Definition}
\newtheorem{example}[theorem]{Example}
\newtheorem{assumption}{Assumption}[section]

\theoremstyle{remark}
\newtheorem{remark}{Remark}[section]

\begin{document}


\title{Delayed Impulsive Stabilization of Discrete-Time Systems:\\ A Periodic Event-Triggering Algorithm }

\author{
\name{Kexue Zhang\textsuperscript{a}\thanks{CONTACT Kexue Zhang. Email: kexue.zhang@queensu.ca} and Elena Braverman\textsuperscript{b}}
\affil{\footnotesize \textsuperscript{a}Department of Mathematics and Statistics, Queen's University, Kingston, Ontario K7L 3N6, Canada\\ \textsuperscript{b}Department of Mathematics and Statistics, University of Calgary, Calgary, Alberta T2N 1N4, Canada}
}

\maketitle

\begin{abstract}
This paper studies the problem of event-triggered impulsive control for discrete-time systems. A novel periodic event-triggering scheme with two tunable parameters is presented to determine the moments of updating impulsive control signals which are called event times. Sufficient conditions are established to guarantee asymptotic stability of the resulting impulsive systems. It is worth mentioning that the event times are different from the impulse times, that is, the control signals are updated at each event time but the actuator performs the impulsive control tasks at a later time due to time delays. The effectiveness of our theoretical result with the proposed scheme is illustrated by three examples.
\end{abstract}


\section{Introduction}

Discrete-time systems are frequently encountered and widely used in various areas (see, e.g., \cite{KO:1995,HB-BGK:1993,JS:2006}), such as digital control, digital signal processing and communication, and optimization. The method of impulsive control activates the control inputs to discrete-time systems only at some ideal discrete-time moments, instead of every time step continuously, and maintains zero inputs between two consecutive control executions. The advantage of impulsive control is to minimize the energy consumption of executing the control tasks. The impulsive control inputs are normally called \textit{impulses} and the associated control-execution moments are refereed as \emph{impulse times}. Recent years have witnessed increasing interest in the study of impulsive control for discrete-time systems (see, e.g., \cite{XL-KZ:2019,ZHG-NL:2010,WHC-XL-WXZ:2015}).

Most of the existing results on impulsive control for discrete-time systems are based on time-triggered impulses (see, e.g., \cite{YZ-JS-GF:2009,BL-DJH-ZS:2018,BL-HJM:2007,BL-DJH:2014}), that is, the impulses are triggered by a clock. More precisely, the impulses are pre-scheduled and independent of the real-time system states. Recently, self-triggered impulsive control has been investigated in \cite{TMPG-WPMH:2015,EH-DEQ-HS-KHJ:2012}, where at each impulse time the next one is determined based on the information available at that impulse time. On the other hand, event-triggered impulsive control only activates the impulse inputs when needed, and the updating of the impulsive control inputs is triggered by an event that occurs when a certain measurement of system states violates a well-designed threshold. Hence, event-triggered impulsive control is expected to be more effective in terms of the efficiency improvement on control implementations when compared with both the time-triggered impulsive control and the conventional event-triggered control which works in a sample-and-hold (or zero-order-hold) fashion (see, e.g., \cite{AE-VD-KK:2010}), and also inherits the advantage of self-triggered impulsive control in running the system open-loop between impulse times (see \cite{WPMHH-KHJ-PT:2012} for a detailed discussion on both self-triggered and event-triggered control).

Nevertheless, very few results have been reported on event-triggered impulsive control for discrete-time systems (see, e.g., \cite{BL-DJH-CZ-ZS:2018,BL-DJH-ZS-JH:2019,HL-JF-XL-LR-TH:2020}). { A major difficulty in the study of event-triggered impulsive control for discrete-time systems is to ensure the consecutive impulsive controls are separated by at least two time steps so that the event-triggered impulsive control can be distinguished from the conventional feedback control, and then the advantage of event-triggered control method on saving the energy consumption due to control updating can be preserved.} Stabilization of nonlinear discrete-time systems was studied in~\cite{BL-DJH-CZ-ZS:2018} via event-triggered impulsive control, and time delays are considered in the impulse inputs but not in the actuator and sensor pair. Moreover, the impulsive control method in~\cite{BL-DJH-CZ-ZS:2018} is to reset system states at each impulse time, which is different from the conventional feedback control approach, and hence different from the impulsive control method investigated in this research. Exponential stabilization of discrete-time systems with time delays was investigated in~\cite{BL-DJH-ZS-JH:2019}, and the impulse times are determined by a novel event-triggering algorithm which has enforced lower and upper bounds for the inter-event times. Recently, this idea of event-triggering mechanism was generalized to deal with the impulsive synchronization for discrete-time coupled neural networks with stochastic perturbations and multiple delays in~\cite{HL-JF-XL-LR-TH:2020}. It should be clarified that, in all the above mentioned results, an upper bound of the inter-event times is prescribed which makes the result conservative in the sense that more control updates are potentially triggered than needed; multiple levels of events need to be detected at every time step in order to determine the time for control updates; the event times coincide with all the impulse times, that is, no time delays between the controller and actuator pair are considered within these event-triggering schemes. It can be seen that the study of event-triggered impulsive control for discrete-time systems is undergoing early-stage investigation which motivates this research.

Inspired by the above discussion and the triggering condition with experiential convergent threshold in \cite{MM-AA-PT:2010,RP-PT-DN-AA:2014}, we design a new periodic event-triggered impulsive control method for discrete-time systems. Due to the communication delays between the sensor and actuator, the event times when to update the control signals are different from the impulse times. Based on a Lyapunov function, we design a periodic event-triggering algorithm to determine the times for control updating, that is, the system states are detected periodically (not necessarily at every discrete moment), and the impulsive control signals are updated once the Lyapunov function exceeds a time-dependent threshold at some periodic sampling moment. The impulsive control mechanism with the proposed event-triggering algorithm ensures the consecutive event times are separated by at least two units of discrete times, which distinguishes the event-triggered impulsive control from the feedback control that requires the control actuation at every discrete-time moment. Compared with the existing results on event-triggered impulsive control (see, \cite{BL-DJH-CZ-ZS:2018,BL-DJH-ZS-JH:2019,HL-JF-XL-LR-TH:2020}), the proposed algorithm is simple to implement as our event-triggering condition only requires the information of the Lyapunov function at the periodic sampling times, and time delays are considered in the proposed event-triggered impulsive control method. {Compared with the recent results on event-triggered control for discrete-time delay systems in \cite{KZ-EB-BG:2023}, the control inputs between consecutive control updates are zeros in the proposed event-triggered impulsive control method, while the event-triggered control mechanism in \cite{KZ-EB-BG:2023} requires the control inputs remain unchanged, most likely nonzero, between event times. Therefore, the event-triggered impulsive control method has the advantage over the event-triggered feedback control method in saving the energy consumption due to execution of control tasks.}

The rest of the paper is organized as follows. We formulate the control problem and propose our event-triggering scheme in Section~\ref{Sec2}. The main result is introduced in Section~\ref{Sec3}. Three examples are presented in Section~\ref{Sec5}. Finally, conclusions are drawn in Section~\ref{Sec6}.

\textbf{Notation.} Let $\mathbb{Z}$ denote the set of integers, $\mathbb{Z}^+$ the set of nonnegative integers, $\mathbb{N}$ the set of positive integers, $\mathbb{R}$ the set of real numbers, $\mathbb{R}^+$ the set of nonnegative reals, and $\mathbb{R}^n$ the $n$-dimensional real space equipped with the Euclidean norm denoted by $\|\cdot\|$. For an $n\times n$ matrix $A$, we use $\|A\|$ to represent its induced matrix norm. Let $D=\textrm{diag}(d_1,d_2,...,d_n)$ denote the diagonal $n\times n$ matrix with diagonal entries $d_1,d_2,...,d_n$. A continuous function $\gamma:\mathbb{R}^+\rightarrow\mathbb{R}$ is said to be of class $\mathcal{K}$ and we write $\gamma\in\mathcal{K}$, if $\gamma$ is strictly increasing and equals to zero at zero. For a function $\gamma\in\mathcal{K}$, we let $\gamma^{-1}$ represent the inverse function of $\gamma$. Let $\delta:\mathbb{Z}\rightarrow \mathbb{Z}$ denote the discrete-time unit sample (or unit impulse) function defined as
\begin{eqnarray}
\delta[k]=\left\{\begin{array}{ll}
1, \textrm{~if~} k=0; \cr
0, \textrm{~otherwise}.\nonumber
\end{array}\right.
\end{eqnarray}
Given $R>0$, $\mathcal{B}(R)$ denotes the open ball in $\mathbb{R}^n$ centered at the origin with radius $R$, that is, $\mathcal{B}(R)=\{x\in\mathbb{R}^n: \|x\|<R\}$.

\section{Problem Formulation}\label{Sec2}

Consider the following discrete-time control system:
\begin{eqnarray}\label{sys}
\left\{\begin{array}{ll}
x(k+1)=f(x(k),u(k)) \cr
x(0)=x_0
\end{array}\right.
\end{eqnarray}
where $k\in\mathbb{Z}^+$, $x(k)\in\mathbb{R}^n$ is the system state, $x_0\in\mathbb{R}^n$ is the initial state, and $u: \mathbb{Z}^+\rightarrow \mathbb{R}^m$ is the control input. The nonlinear function $f:\mathbb{R}^n\times\mathbb{R}^m \rightarrow\mathbb{R}^n$ satisfies $f(0,0)=0$, then system \eqref{sys} admits the trivial solution.

\begin{definition}[Stability~{\citep{KO:1995}}]
The trivial solution of system \eqref{sys} is said to be 
\begin{itemize}
\item \textbf{stable}, if, for any $\varepsilon>0$, there exists a positive constant $\sigma:=\sigma(\varepsilon)$, such that $\|x_0\|<\sigma \Rightarrow\|x(k)\|< \varepsilon$ for all $k\geq 0$;

\item \textbf{asymptotically stable}, if the trivial solution of system \eqref{sys} is stable, and there exists a constant $\sigma>0$ such that $\|x_0\|<\sigma \Rightarrow \lim_{k\rightarrow\infty}\|x(k)\|=0$,
\end{itemize}
where $x(k)$ is the solution of system~\eqref{sys}.
\end{definition}

In this study, we consider the following state feedback control 
\begin{equation}\label{pulse}
u(k)=\sum_{i\in\mathbb{N}} \mathbf{k}(x(k_i)) \delta[k-(k_i+\Gamma)]
\end{equation}
where $\mathbf{k}:\mathbb{R}^n\rightarrow \mathbb{R}^m$ is the feedback control law with $\mathbf{k}(0)=0$, and $\Gamma\geq 0$ is the time delay. The times $\{k_i\}_{i\in\mathbb{N}}$ are the moments when control $u$ is updated and are to be determined by a certain event that occurs when the measurement of the system states violates a triggering condition to be designed later. It can be seen from control~\eqref{pulse} that the control input is $\mathbf{k}(x(k_i))$ at the impulse time $k_i+\Gamma$ for $i\in\mathbb{N}$; otherwise, the input is zero. Hence, feedback control~\eqref{pulse} is normally called \textit{impulsive control}. Closed-loop system~\eqref{sys} with impulsive control~\eqref{pulse} can be written as a discrete-time impulsive system:
\begin{eqnarray}\label{CLsys}
\left\{\begin{array}{ll}
x(k+1)=g(x(k)),  ~k\not= k_i+\Gamma \cr
x(k_i+\Gamma+1)=f\left(x(k_i+\Gamma),\mathbf{k}(x(k_i))\right),  ~i\in\mathbb{N} \cr
x(0)=x_0 
\end{array}\right.
\end{eqnarray}
where $g(x):=f(x,0)$ for $x\in\mathbb{R}^n$. To introduce our event-triggering algorithm, we make the following assumption on system~\eqref{CLsys}.

\begin{assumption}\label{Assumption1}
There exist functions $V:\mathbb{R}^n\rightarrow \mathbb{R}^+$, $\alpha,\beta\in\mathcal{K}$, and constants $c\geq 1$, $\rho>0$ such that, for any $x\in\mathbb{R}^n$, the following conditions are satisfied
\begin{itemize}
\item[$(A1)$] $\alpha(\|x\|)\leq V(x)\leq \beta(\|x\|)$;

\item[$(A2)$] $V(g(x))\leq c V(x)$;

\item[$(A3)$] $V(f(g^{\Gamma}(x),\mathbf{k}(x)))\leq \rho V(x)$ where $g^{\Gamma}:=\underbrace{g\circ g \circ ... \circ g}_\mathrm{\Gamma~\text{times}}$ is the $\Gamma$ times composition of function $g$.

\end{itemize}
\end{assumption}

\begin{remark}
When $k\not=k_i+\Gamma$, condition (A2) describes the dynamics of the uncontrolled system. If $c<1$, then the uncontrolled system is asymptotically stable. Thus, we only consider the case of $c\geq 1$ in (A2). Condition (A3) characterizes the impulse effect on the Lyapunov function $V$. To be more specific, the relation between the values of $V$ at event time $k_i$ and time $k_i+\Gamma+1$ after the impulse time is quantified by constant $\rho$. See Section~\ref{Sec5} for demonstrations of how to derive $\rho$. {Local Lipschitz conditions on $f$, $g$, $\mathbf{k}$, and $V$ will ensure that both $V(g)$ and $V(f(g^{\Gamma},\mathbf{k}))$ are locally Lipschitz continuous in $x$, and then Assumption~\ref{Assumption1} can be satisfied on a compact set so that Theorem~\ref{Th1} can be applied.}
\end{remark}

Now we are in the position to introduce the event trigger. The event time sequence $\{k_i\}_{i\in\mathbb{N}}$ is determined by the following trigger
\begin{align}\label{trigger}
k_{i+1}
=&\left\{\begin{array}{ll}
\inf\{j\Delta> 0:V(x(j\Delta))>a(1-b)^{j\Delta}\},&\textrm{~if~} i=0 \cr
\inf\{j\Delta> k_i+\Gamma :V(x(j\Delta))>a(1-b)^{j\Delta}\},&\textrm{~otherwise}\cr
\end{array}\right.
\end{align}
where $j\in\mathbb{N}$, $a$ and $b$ are positive constants with $b<1$, and $\Delta\in\mathbb{N}$ is the sampling period. It can be observed from trigger~\eqref{trigger} that the \textit{event} 
\[
V(x(k))>a(1-b)^k
\]
is only detected at the sequence of sampling times $\{j\Delta\}_{i\in\mathbb{N}}$ instead of the entire time span, that is, the event is detected periodically. Thus, trigger~\eqref{trigger} is called a \textit{periodic event trigger}. Nevertheless, the event is detected at every discrete moment if $\Delta=1$. With trigger~\eqref{trigger}, impulsive control~\eqref{pulse} works as follows. For any initial state $x_0\in \mathbb{R}^n$ satisfying $V(x_0)<a$, event time $k_1$ is the sampling moment when the graph of $V(x(k))$ goes above the threshold line $a(1-b)^k$. The control $u$ is updated at event time $k_1$ and then executed at the impulse time $k_1+\Gamma$, due to time delay $\Gamma$. The purpose of the impulsive control is to bring the value of $V$ at time $k_1+\Gamma+1$ down below the threshold line. Then, the next event time $k_2$ is the sampling time when $V$ surpasses the threshold again. The above mentioned process is repeated as long as the value of $V$ goes beyond the threshold line at the sampling times. See Figure~\ref{Mechanism} for the demonstration of the proposed event-triggered impulsive control mechanism. It can be seen that, to ensure the validity of event trigger~\eqref{trigger}, it is necessary to guarantee that the value of $V$ is not bigger than the threshold after each impulse.

\begin{remark} Time delay $\Gamma$ in the impulses can be understood as follows. An event is detected at time
$k_i$ and the corresponding measurement is $x(k_i)$, which arrives to the controller at
time $k_i + \Gamma_1$ (where $\Gamma_1$ the \textit{sensor-controller} delay), and due to the existence of \textit{controller-actuator} delay $\Gamma_2$,
the control input $u(x(k_i))$ is applied to the plant at time $k_i+\Gamma$, where $\Gamma=\Gamma_1+\Gamma_2$, i.e., the sum of {sensor-controller} delay and {controller-actuator} delay. Moreover, the analysis in this study is also applicable to time-varying delays in the impulse with $\Gamma$ as the upper bound of all these time-varying delays.
\end{remark}

\begin{remark} In this study, we extend the idea of event-triggered impulsive control for continuous-time systems in~\cite{KZ-EB:2022} to deal with stabilization of discrete-time systems. However, the impulsive control method in this research is different from that in ~\cite{KZ-EB:2022} in the following sense. For continuous-time systems, the impulsive control inputs are unbounded and lead to state jumps at each impulse time. The impulsive control for discrete-time systems is a typical feedback control, and the control inputs are finite at impulse times and zeros at the non-impulse times. { Another difference lies on the requirement of the event times. For continuous-time systems, it is essential to ensure the inter-event times are lower bounded by a positive quantity so that Zeno behavior, a phenomenon of infinite many impulses over a finite time interval, will not occur. For discrete-time systems, Zeno behavior can be naturally excluded because of the discrete-time dynamic evolution. However, to distinguish from the conventional feedback control, the inter-event times for discrete-time control systems should be bounded by at least two time steps from below. To ensure such a lower bound for the inter-event times is one of the major challenges in the study of event-triggered control problems for discrete-time systems. In the following section, sufficient conditions will be established to guarantee the lower bound of the inter-event times is not less than $\Gamma+2$.} 
\end{remark}

\begin{figure}[t!]
\centering
\includegraphics[width=3.2in]{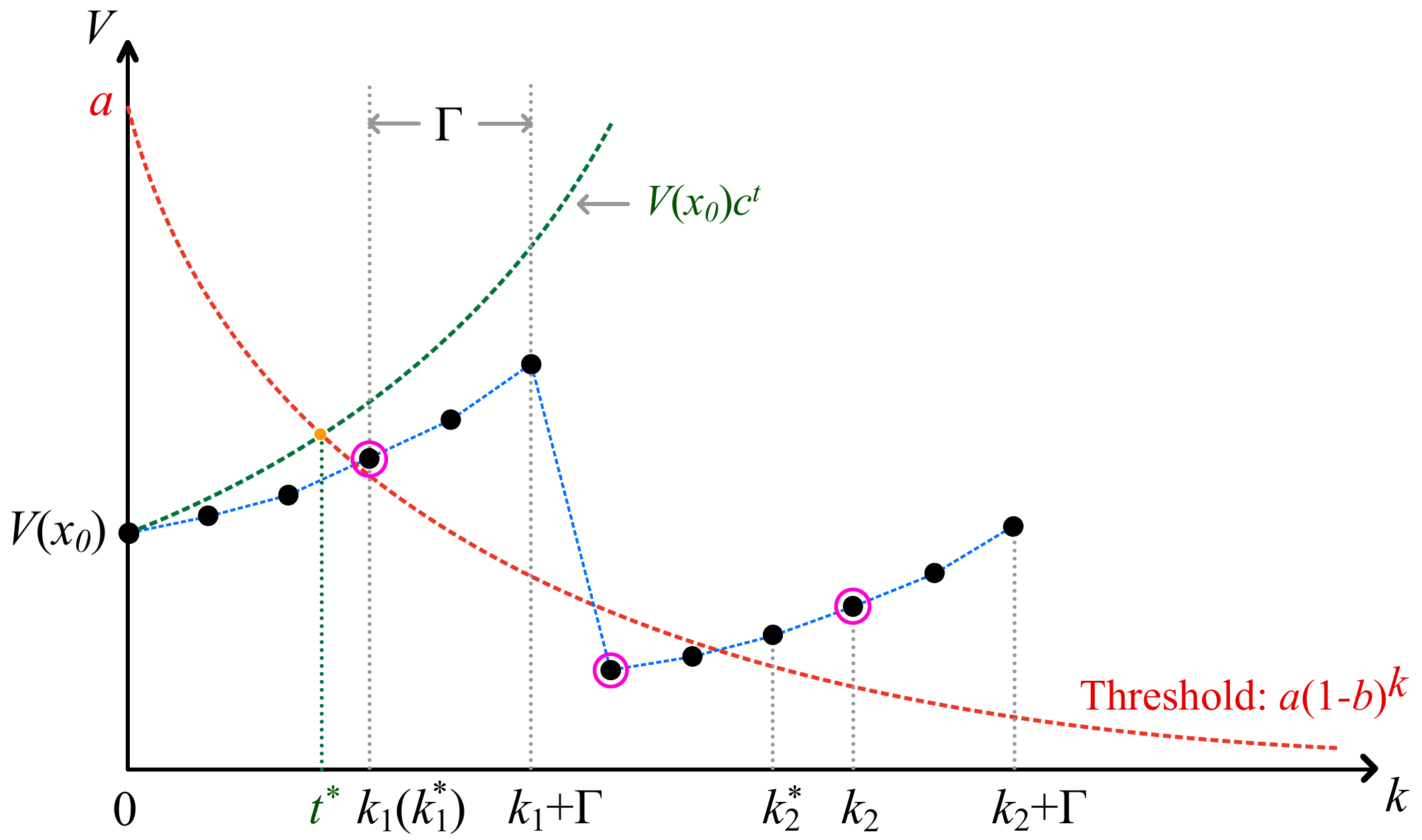}
\caption{A mechanism demonstration of the proposed event-triggered impulsive control with $\Delta=3$ and $\Gamma=2$. The black dots represent the Lyapunov function $V$ at the discrete moments in the interval $[0,k_2+\Gamma]$, which are traced by the blue dotted line. The values of $V$ at the sampling times are indicated by magenta circles. Graphs of functions $V(x_0)c^t$ and $a(1-b)^t$ for $t\in\mathbb{R}^+$ intersect at $t^*$ which plays an important role in the proof of Theorem~\ref{Th1} on stability of system~\eqref{CLsys}. From the definition of $k^*_i$ in the proof of Theorem~\ref{Th1}, we can observe that $k^*_1=k_1$ and $k^*_2=k_2-1$ in this demonstration.}
\label{Mechanism}
\end{figure}

\section{Main result}\label{Sec3}
In this section, we introduce the main result to ensure the validity of trigger~\eqref{trigger} and asymptotic stability of closed-loop system~\eqref{CLsys}.

\begin{theorem}\label{Th1}
Suppose Assumption~\ref{Assumption1} holds for system~\eqref{CLsys}, and the sequence of event times $\{k_i\}_{i\in\mathbb{N}}$ is determined by trigger~\eqref{trigger}. If
\begin{equation}\label{condition}
\frac{\rho c^{\Delta}}{(1-b)^{\Delta+\Gamma+1}}\leq 1
\end{equation}
then, for any initial state $x_0\in \mathcal{B}\left(\beta^{-1}(a)\right)$, the trivial solution of system~\eqref{CLsys} is asymptotically stable. Furthermore, the inter-event times  $\{k_{i+1}-k_i\}_{i\in\mathbb{N}}$ are lower bounded by $\Gamma+2$, that is, $k_{i+1}-k_i \geq \Gamma+2$ for all $i\in\mathbb{N}$.
\end{theorem}
\begin{proof} Let $x(k):=x(k,0,x_0)$ denote the solution of system~\eqref{sys}. 
From trigger~\eqref{trigger}, we can see that $k_1$ is the first sampling time so that $V(x(k))>a(1-b)^k$. Then we can conclude that there exists a $k^*_1\in (k_1-\Delta,k_1]$ such that
\begin{equation}\label{k*1-1}
V(x(k^*_1))>a(1-b)^{k^*_1},
\end{equation}
\begin{equation}\label{k*1-2}
V(x(k^*_1-1))\leq a(1-b)^{k^*_1-1},
\end{equation}
and
\begin{equation}\label{k*1-3}
V(x(k))>a(1-b)^{k} \textrm{~for~all~} k\in[k^*_1,k_1],
\end{equation}
that is, 
$$k^*_1=\max\left\{k\in (k_1-\Delta,k_1]: V(x(k))\leq a(1-b)^k\right\}+1.$$
 See Figure~\ref{Mechanism} for a demonstration of $k^*_1$.

For $k=k_1+\Gamma+1$, we have
\begin{align}\label{afterpulse1}
 V(x(k_1+\Gamma+1))
=& V\left(f(x(k_1+\Gamma),\mathbf{k}(x(k_1)))\right) \cr
=& V\left(f\left(g^{\Gamma}(x(k_1)),\mathbf{k}(x(k_1))\right)\right) \cr
\leq & \rho V(x(k_1)) \cr
\leq & \rho c^{k_1-(k^*_1-1)} V(x(k^*_1-1))\cr
\leq & \rho c^{\Delta} a (1-b)^{k^*_1-1} \frac{(1-b)^{k_1+\Gamma+1-(k^*_1-1)}}{(1-b)^{k_1+\Gamma+1-(k^*_1-1)}}\cr
\leq & \frac{\rho c^{\Delta}}{(1-b)^{\Delta+\Gamma+1}} a (1-b)^{k_1+\Gamma+1}
\end{align}
where we used the dynamics of system~\eqref{CLsys} when $k=k_1+\Gamma$ in the first equality, and then the dynamics from time $k_1$ to $k_1+\Gamma$ in the second equality. The first and second inequalities of~\eqref{afterpulse1} follow from (A3) and (A2) of Assumption~\ref{Assumption1}, respectively. The last two inequalities of~\eqref{afterpulse1} follow from the fact that $k_1-k^*_1\leq \Delta-1$. It can be seen that if~\eqref{condition} holds, then $V$ is not larger than the threshold at $k=k_1+\Gamma+1$. Repeating the above discussion as long as $V$ goes above the threshold line, we can get that, for each event time $k_i$, there exists a $k^*_i\in (k_i-\Delta, k_i]$ so that
\begin{equation}\label{k*i-1}
V(x(k^*_i))>a(1-b)^{k^*_i},
\end{equation}
\begin{equation}\label{k*i-2}
V(x(k^*_i-1))\leq a(1-b)^{k^*_i-1},
\end{equation}
and
\begin{equation}\label{k*i-3}
V(x(k))>a(1-b)^{k} \textrm{~for~all~} k\in[k^*_i,k_i].
\end{equation}
Moreover, we can conclude that~\eqref{condition} guarantees $V(x(k_i+\Gamma+1))\leq a(1-b)^{k_i+\Gamma+1}$ for all the event times $k_i$. Hence, trigger~\eqref{trigger} is valid with condition~\eqref{condition}, that is, $V$ stays on or below the threshold line right after each impulse.

Next, we show the attractivity of system~\eqref{CLsys}. For any $k\in [0,k^*_1-1]$, there is a $j\in\mathbb{N}$ so that $(j-1)\Delta \leq k <j\Delta$. We have $V(x((j-1)\Delta))\leq a (1-b)^{(j-1)\Delta}$, because $k_1$ is the first sampling moment when $V$ is above the threshold line. We then conclude that
\begin{align}\label{attractivity1}
V(x(k))&\leq c^{k-(j-1)\Delta} V(x((j-1)\Delta)) \cr
       &\leq c^{\Delta-1} a (1-b)^{(j-1)\Delta} \frac{(1-b)^{k-(j-1)\Delta}}{(1-b)^{k-(j-1)\Delta}} \cr
       &\leq \left(\frac{c}{1-b}\right)^{\Delta-1} a (1-b)^k
\end{align}
for all $k\in [0,k^*_1-1]$, where we used (A2) of Assumption~\ref{Assumption1} and the fact that $k-(j-1)\Delta\leq \Delta-1$. 

Similarly, for any $k\in [k_{i-1}+\Gamma+1,k^*_i-1]$ with $i>1$, there is a $j\in\mathbb{N}$ so that $(j-1)\Delta \leq k <j\Delta$. Let 
\[
\bar{k}=\max\left\{k_{i-1}+\Gamma+1,(j-1)\Delta\right\},
\]
then, $V(x(\bar{k}))\leq a (1-b)^{\bar{k}}$ and
\begin{align}\label{attractivity2}
V(x(k))&\leq c^{k-\bar{k}} V(x(\bar{k})) \cr
       &\leq c^{\Delta-1} a (1-b)^{\bar{k}} \frac{(1-b)^{k-\bar{k}}}{(1-b)^{k-\bar{k}}} \cr
       &\leq \left(\frac{c}{1-b}\right)^{\Delta-1} a (1-b)^k
\end{align}
for all $k\in [k_{i-1}+\Gamma+1,k^*_i-1]$, where we used the fact that $k-\bar{k}\leq \Delta-1$ in the last two inequalities.

For any $k\in[k^*_i,k_i+\Gamma]$ with $i\in\mathbb{N}$, we get from (A2) of Assumption~\ref{Assumption1} that
\begin{align}\label{attractivity3}
V(x(k)) &\leq c^{k-(k^*_i-1)} V(x(k^*_i-1)) \cr
        &\leq c^{k-k_i+k_i-(k^*_i-1)} a(1-b)^{k^*_i-1} \frac{(1-b)^{k-(k^*_i-1)}}{(1-b)^{k-(k^*_i-1)}} \cr
        &\leq \left(\frac{c}{1-b}\right)^{\Gamma+\Delta} a (1-b)^k.
\end{align}
Hence, we conclude from~\eqref{attractivity1},~\eqref{attractivity2}~and~\eqref{attractivity3} that
\begin{equation}\label{attractivity}
V(x(k))\leq \left(\frac{c}{1-b}\right)^{\Gamma+\Delta} a (1-b)^k \textrm{~~for~all~} k\geq 0,
\end{equation}
which implies the attractivity of system~\eqref{CLsys}.

Lastly, we show stability of system~\eqref{CLsys}. The fact that $x_0\in \mathcal{B}\left(\beta^{-1}(a)\right)$ and  Assumption~\ref{Assumption1}(A1) imply $V(x_0)\leq\beta(\|x_0\|)<a$. Then, we can get from the strict monotonicity of functions $V(x_0) c^t$ and $a(1-b)^t$ with $t\in\mathbb{R}^+$ that there exists a unique $t^*>0$ such that $V(x_0)c^{t^*}=a(1-b)^{t^*}$, which then implies that
\begin{equation}\label{defnition_t*}
t^*= \frac{\ln\left(\frac{a}{V(x_0)}\right)}{\ln\left(\frac{c}{1-b}\right)}.
\end{equation}
Note that $t^*$ may not be an integer. For $0\leq k \leq k^*_1-1$, we consider the following two scenarios.
\begin{itemize}
\item If $k\leq t^*$, then we get from (A2) of Assumption~\ref{Assumption1} that $V(x(k)) \leq V(x_0) c^k \leq V(x_0) c^{t^*}$.

\item If $k>t^*$, then we derive from~\eqref{attractivity1} that
\begin{align*}
V(x(k))&\leq \left(\frac{c}{1-b}\right)^{\Delta-1} a (1-b)^k \cr
       &<\left(\frac{c}{1-b}\right)^{\Delta-1} a(1-b)^{t^*}.
\end{align*}

\end{itemize}
Hence, we can conclude from the above two scenarios with the definition of $t^*$ that
\begin{equation}\label{k_1-1}
V(x(k))\leq \left(\frac{c}{1-b}\right)^{\Delta-1} a (1-b)^{t^*} \textrm{~for~all~} 0\leq k\leq k^*_1-1.
\end{equation}
For $k=k^*_1-1$, more precisely, we have the following estimations:
\begin{itemize}
\item if $k^*_1-1 \leq t^*$, Assumption~\ref{Assumption1}(A2) implies
\[
V(x(k^*_1-1))\leq V(x_0) c^{k^*_1-1} \leq V(x_0) c^{t^*};
\]
\item if $k^*_1-1 > t^*$, the definition of $k^*_1$ implies
 \[
 V(x(k^*_1-1))\leq a (1-b)^{k^*_1-1} \leq a (1-b)^{t^*}.
 \]
\end{itemize}
Thus,
\begin{equation}\label{k^*_1-1}
V(x(k^*_1-1)) \leq a (1-b)^{t^*}.
\end{equation}
For $k^*_1\leq k \leq k_1+\Gamma$, we can derive from (A2) of Assumption~\ref{Assumption1} and~\eqref{k^*_1-1} that
\begin{align}\label{k_1}
V(x(k)) \leq V(x(k^*_1-1)) c^{k-(k^*_1-1)} \leq a (1-b)^{t^*} c^{\Gamma+\Delta}.
\end{align}
For any $k>k_1+\Gamma$, we get from~\eqref{attractivity} and the fact $k_1+\Gamma>t^*$ that
\begin{align}\label{>k_1+Gamma}
V(x(k))&\leq \left(\frac{c}{1-b}\right)^{\Gamma+\Delta} a (1-b)^k \cr
       &\leq \left(\frac{c}{1-b}\right)^{\Gamma+\Delta} a (1-b)^{t^*}.
\end{align}

We then can conclude from~\eqref{k_1-1},~\eqref{k_1} and~\eqref{>k_1+Gamma} that
\begin{align}\label{stability}
\alpha(\|x(k)\|) &\leq V(x(k)) \cr
                 &\leq a(1-b)^{t^*} \left(\frac{c}{1-b}\right)^{\Gamma+\Delta} \cr
                 &  =  a \left(\frac{c}{1-b}\right)^{\Gamma+\Delta} \left( \frac{V(x_0)}{a} \right)^{{-\ln(1-b)}/{\ln\left(\frac{c}{1-b}\right)}} \cr
                 &\leq a \left(\frac{c}{1-b}\right)^{\Gamma+\Delta} \left( \frac{\beta(\|x_0\|)}{a} \right)^{{-\ln(1-b)}/{\ln\left(\frac{c}{1-b}\right)}}
\end{align}
for all $k\in\mathbb{N}$, where we used the definition of $t^*$ in the equality and (A1) of Assumption~\ref{Assumption1} in both the first and the last inequality. Therefore, for any $\varepsilon>0$, there exists 
\begin{align*}
\sigma=\beta^{-1}\left( a \left(  \frac{a \left(\frac{c}{1-b}\right)^{\Gamma+\Delta} }{\alpha(\varepsilon)}  \right)^{ {\ln\left( \frac{c}{1-b} \right)}/{\ln(1-b)}}   \right)
\end{align*}
such that, for any $\|x_0\|< \min\left\{\sigma,\beta^{-1}(a)\right\}$, we can establish from~\eqref{stability} that $\|x(k)\|<\varepsilon$ for all $k\in\mathbb{N}$, which implies stability of system~\eqref{CLsys}.

To show the lower bound of the inter-event times, we can observe that condition~\eqref{condition} ensures that the value of the Lyapunov function $V$ does not exceed the threshold right after each impulse, that is, $V(x(k_i+\Gamma+1))<a (1-b)^{k_i+\Gamma+1}$. Then, the next possible event time will be bigger than $k_i+\Gamma+1$. Therefore, $k_{i+1}>k_i+\Gamma+1$, i.e, $k_{i+1}-k_i\geq \Gamma+2$ for $i\in\mathbb{N}$.
\end{proof}

\begin{remark}\label{remarkth}
From the proof of Theorem~\ref{Th1}, we can see that parameter $b$ corresponds to the convergence speed of the Lyapunov function $V$. Setting $b$ large increases the speed of convergence at the cost of the event being triggered more frequently. Having $a$ large with fixed value of $b$ provides more time for the Lyapunov function $V$ to evolve from $V(x_0)$ to $V(x(k_1))$, that is, the first event triggering arrives later. However, the frequency of the event occurrence barely changes since the convergence speed of $V$ is unchanged. For impulsive control system~\eqref{CLsys}, parameter $a$ can be chosen large enough so that the initial state $x_0\in \mathcal{B}\left(\beta^{-1}(a)\right)$, which ensures the evolution of $V(x(k))$ starts below the threshold. Select a desired convergence rate $b$ for the threshold line, and then the impulsive control law $\mathbf{k}$ can be designed according to inequality~\eqref{condition} with Assumption~\ref{Assumption1}(A3). { It should be noted that the inter-event times are lower bounded by $\Gamma+2$, and the delay $\Gamma$ in the impulse allows the lower bound to be bigger than 2. However, a high control gain will be expected to compensate the delay effects in the impulses.}
\end{remark}

\section{Examples}\label{Sec5}
In this section, three examples are investigated to demonstrate our theoretical result.

\begin{example}\label{example1}
Consider the following positive scalar system
\begin{eqnarray}\label{nsys}
\left\{\begin{array}{ll}
x(k+1)=A_1 x(k) + A_2 \tanh(x(k)) + B u(k) \cr
x(0)=x_0
\end{array}\right.
\end{eqnarray}
where $x(k)\in\mathbb{R}$, $A_1=1.02$, $A_2=0.1$, and $B=1.5$. The control $u$ is in the form of~\eqref{pulse} with $\mathbf{k}(x)=K x$ and $\Gamma=1$, where constant $K$ is the feedback control gain, and event times $\{k_i\}_{i\in\mathbb{N}}$ are to be determined by trigger~\eqref{trigger} with Lyapunov function $V(x)=|x|$. To ensure positivity of system~\eqref{nsys}, we assume $A_1+BK>0$, that is, $K>-0.68$.
\end{example}

With $V(x)=|x|$, we have that (A1) of Assumption~\ref{Assumption1} is satisfied with $\alpha(s)=\beta(s)=s$ for $s\in\mathbb{R}^+$, and (A2) holds with $c=A_1+A_2=1.03$ because $f(x,u)=A_1x+A_2\tanh(x)+Bu$ and
\[
V(f(x,0))=|A_1x+A_2\tanh(x)|\leq (A_1+A_2) V(x).
\]
To verify~(A3) of Assumption~\ref{Assumption1}, we can observe that $g(x)=A_1x+A_2\tanh(x)$, and then
\begin{align*}
 V(f(g(x),Kx))
=& \left|A_1g(x)+A_2\tanh(g(x))+BKx\right| \cr
=& \mathrm{sgn}(x) \left(A_1g(x)+A_2\tanh(g(x))+BKx\right) \cr
\leq & \left((A_1+A_2)^2+BK\right) V(x)
\end{align*}
where $\mathrm{sgn}(\cdot)$ is the sign function, and the second equality follows from the fact that $A_1+BK>0$. Hence, (A3) holds with $\rho=(A_1+A_2)^2+BK$. We then can conclude from Theorem~\ref{Th1} that if 
\begin{equation}\label{pcondition}
-\frac{A_1}{B}<K\leq \frac{1}{B}\left( \frac{(1-b)^{\Delta+2}}{(A_1+A_2)^{\Delta}} -(A_1+A_2)^2 \right)
\end{equation}
then the closed-loop system is asymptotically stable. In the simulation, we select $K=-0.45$ and $\Delta=2$ so that~\eqref{pcondition} is satisfied. {Figure~\ref{fig1} shows the trajectories of the closed-loop system with different threshold parameters, and Figure~\ref{fig1u} illustrates the corresponding event-triggered impulsive control inputs.} The mechanism of trigger~\eqref{trigger} implies that each event time $k_i$ is a sampling moment, and $k_i+1$ is an impulse time as $\Gamma=1$. Then, condition~\eqref{condition} enforces $V(x(t))\leq a (1-b)^k$ at time $k_i+\Gamma+1=k_i+2$ which is a sampling time with $\Delta=2$. At the next sampling time $k=k_i+2\Delta$, it is possible for $V$ to be bigger than the threshold. We then can conclude that the inter-event times $\{k_{i+1}-k_i\}_{i\in\mathbb{N}}$ are lower bounded by $2\Delta=4$ which can also be observed in Figure~\ref{fig1.0}.

From the proof of Theorem~\ref{Th1}, we can see that parameter $b$ corresponds to the convergence speed of the Lyapunov function $V$. Setting $b$ large increases the speed of convergence at the cost of the event being triggered more frequently. To be more specific, the event occurs with a higher frequency for larger $b$. See Figure~\ref{fig1} for a comparison of system trajectories with different values of $b$. Having $a$ large with fixed value of $b$ provides more time for the Lyapunov function $V$ to evolve from $V(x_0)$ to $V(x(k_1))$, that is, the first event triggering arrives later. However, the frequency of the event occurrence barely changes since the convergence speed of $V$ is unchanged. See Table~\ref{table} for the demonstration of the above discussion. It can be seen that, over the same time interval, enlarging $b$ increases the number of event times dramatically, while this number decreases slightly with the increased value of $a$.

\begin{figure}
\centering
\subfloat[$a=5$ and $b=0.07$]{%
\resizebox*{7.8cm}{!}{\label{fig1.0}\includegraphics{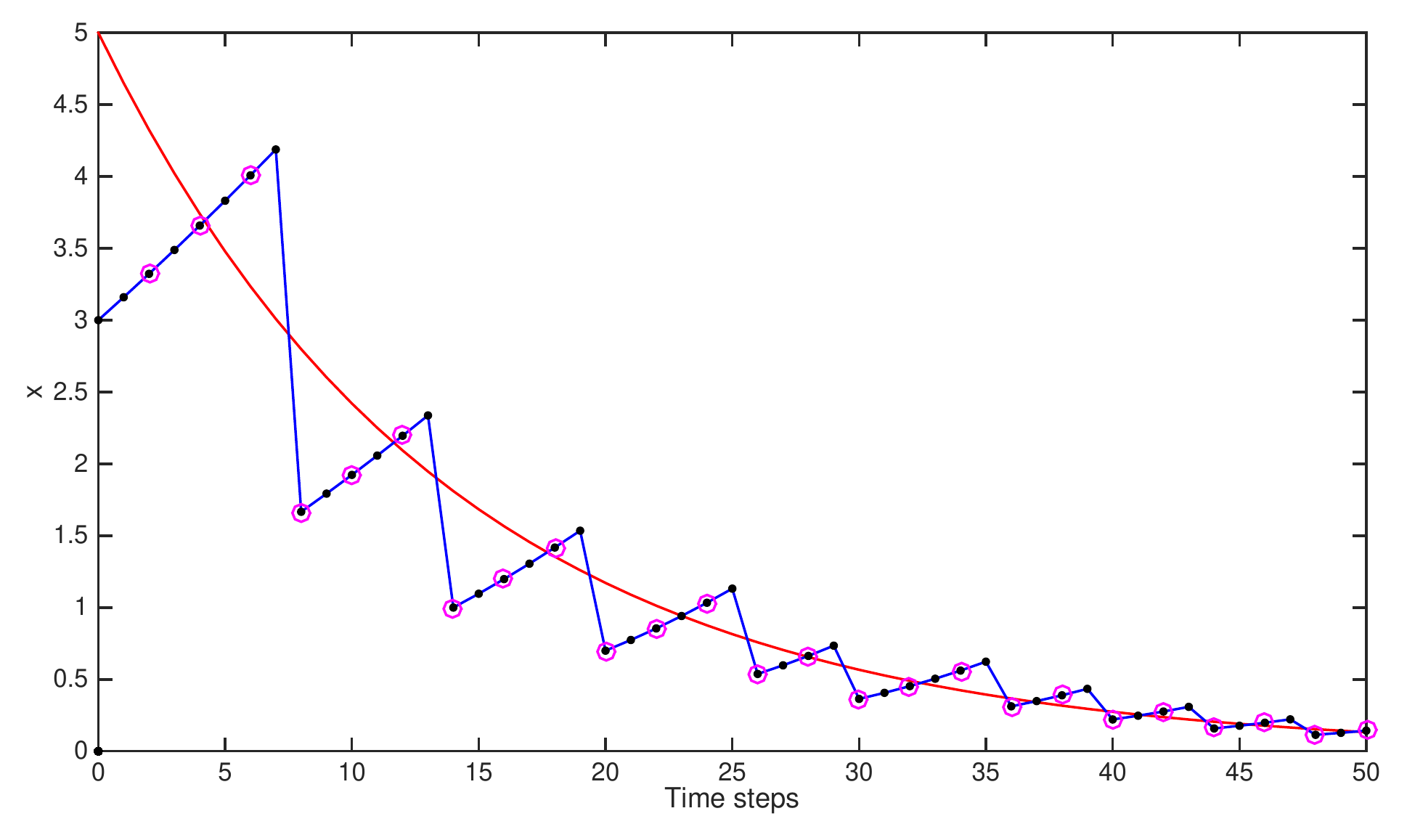}}}\hspace{5pt}
\subfloat[$a=5$ and $b=0.1$]{%
\resizebox*{8cm}{!}{\label{fig1.1}\includegraphics{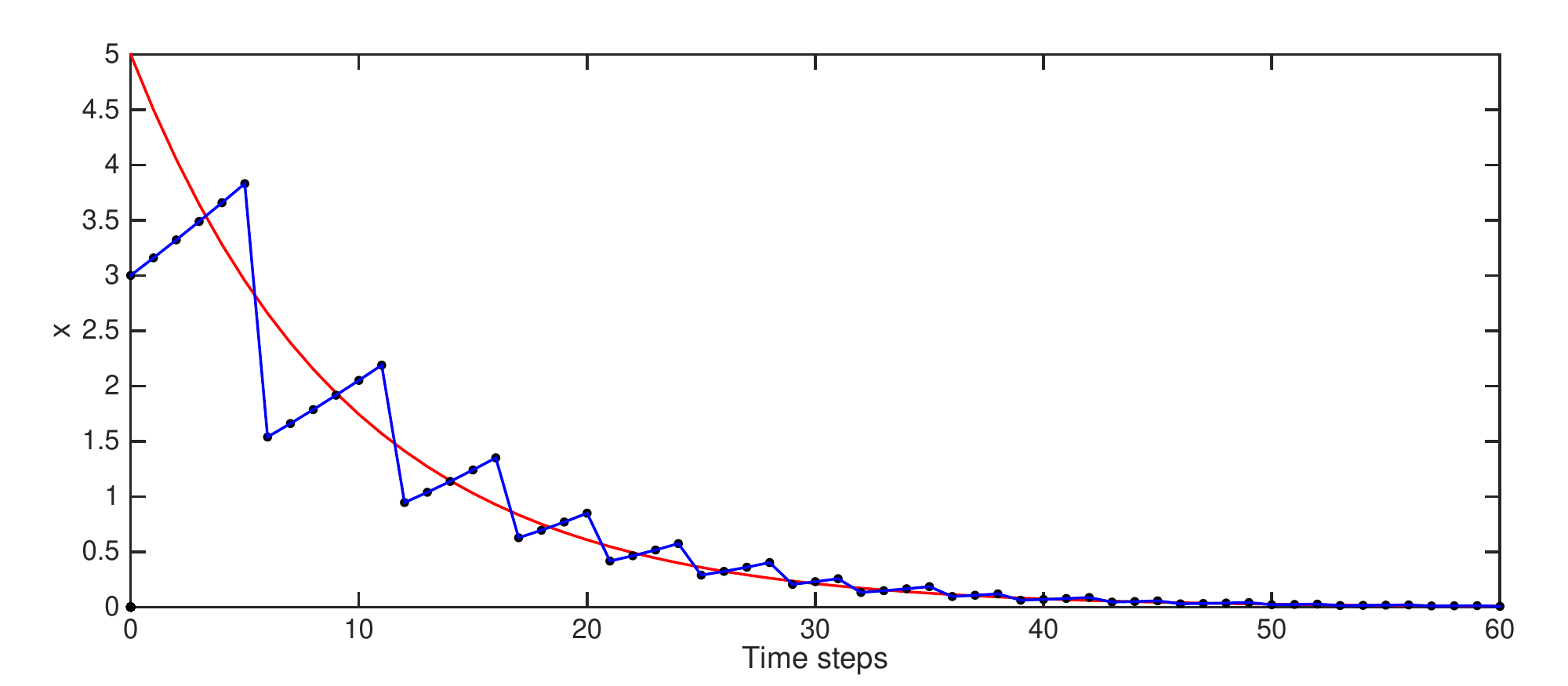}}}\hspace{5pt}
\subfloat[$a=5$ and $b=0.14$]{%
\resizebox*{8cm}{!}{\label{fig1.2}\includegraphics{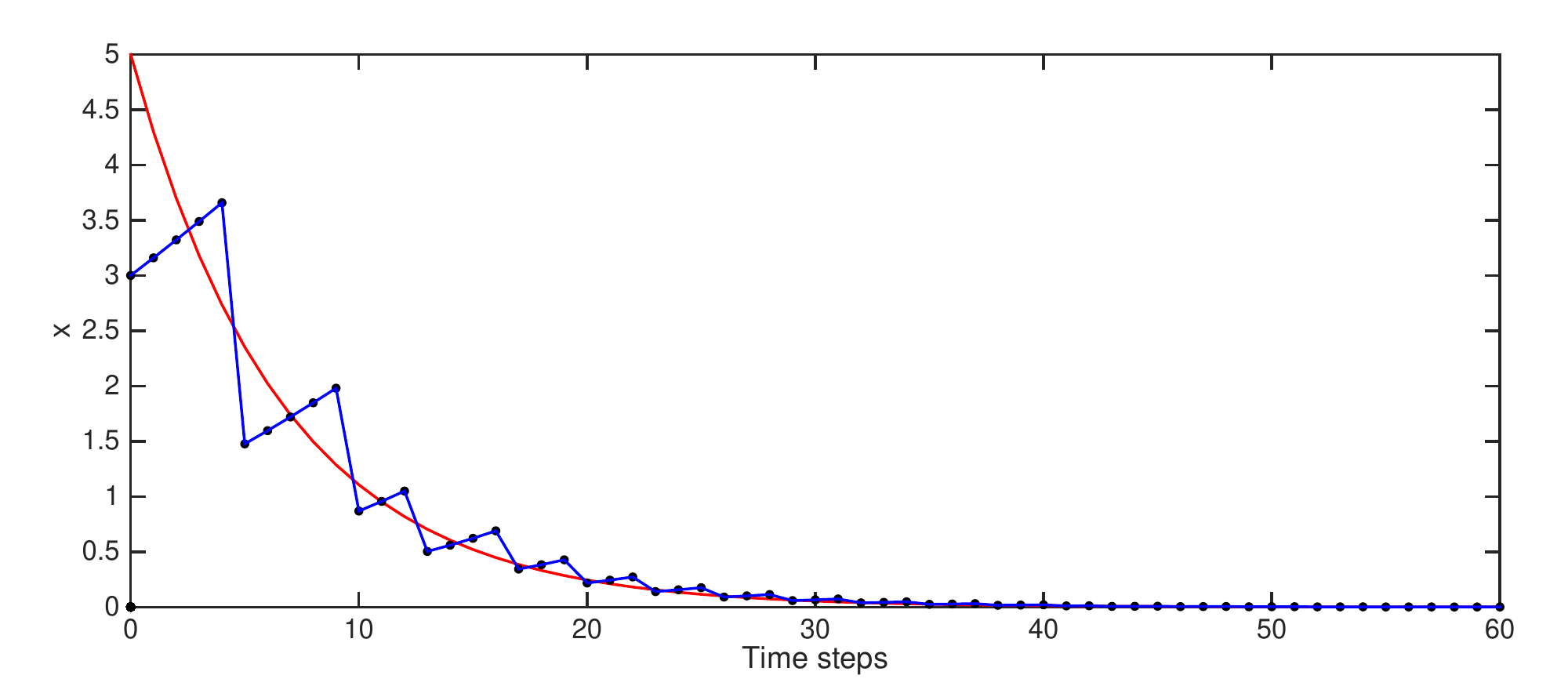}}}
\caption{Simulations of system~\eqref{nsys} with different parameters in the threshold. The red curve represents the threshold line. The values of $V$ at different time steps are indicated by the black dots, which are traced by the blue curves chronologically. The magenta-circled dots correspond to the values of $V$ at the sampling times, and the circled dots lying above the threshold line represent the $V$ values at the event times.} \label{fig1}
\end{figure}

\begin{figure}
\centering
\subfloat[$a=5$ and $b=0.07$]{%
\resizebox*{7.8cm}{!}{\label{fig1.0u}\includegraphics{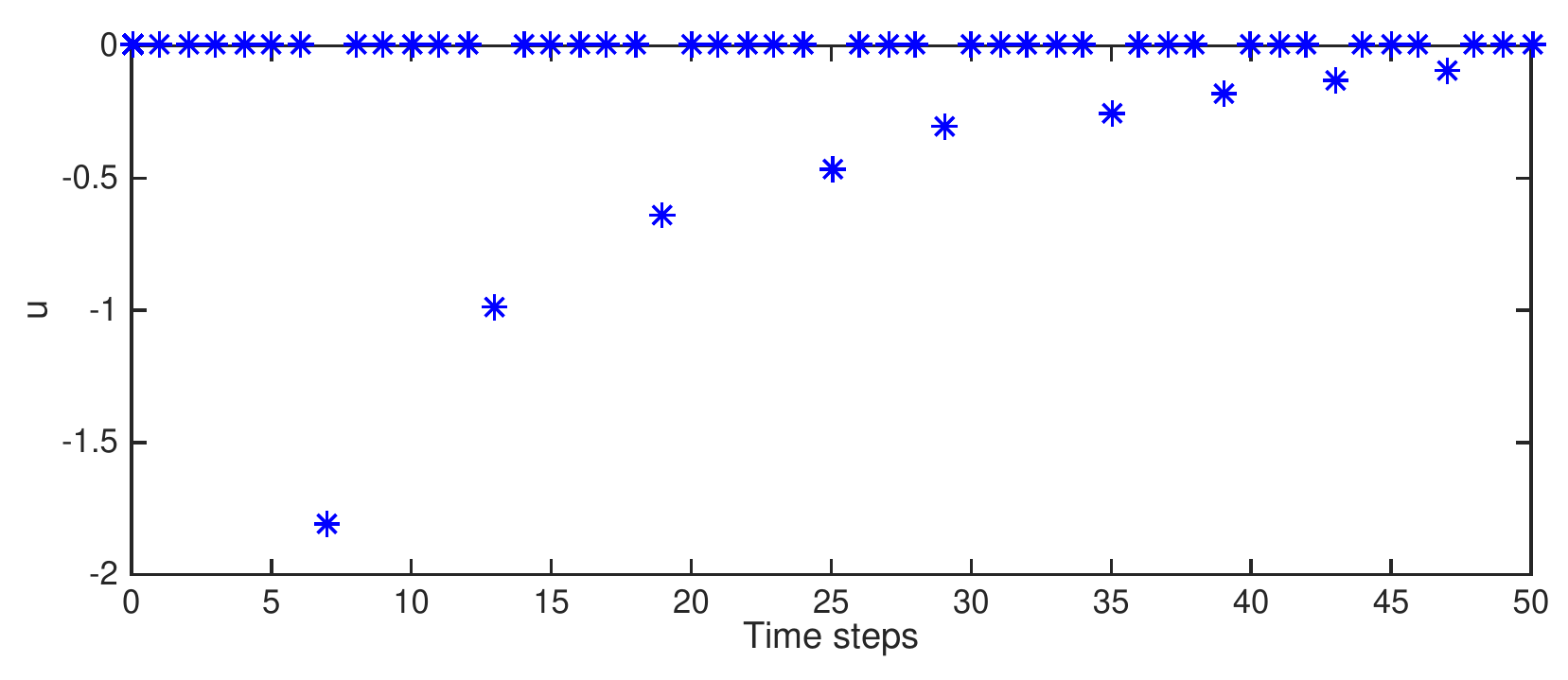}}}\hspace{5pt}
\subfloat[$a=5$ and $b=0.1$]{%
\resizebox*{8cm}{!}{\label{fig1.1u}\includegraphics{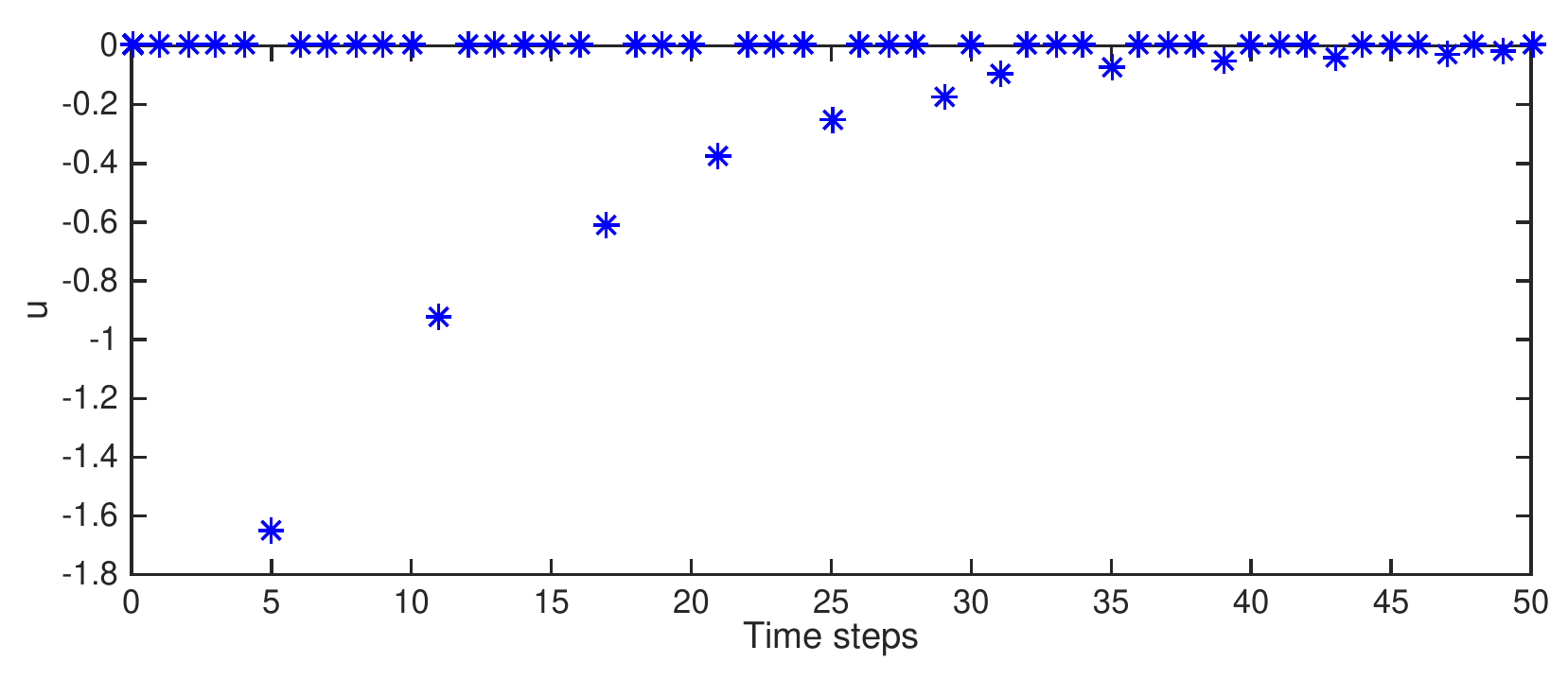}}}\hspace{5pt}
\subfloat[$a=5$ and $b=0.14$]{%
\resizebox*{8cm}{!}{\label{fig1.2u}\includegraphics{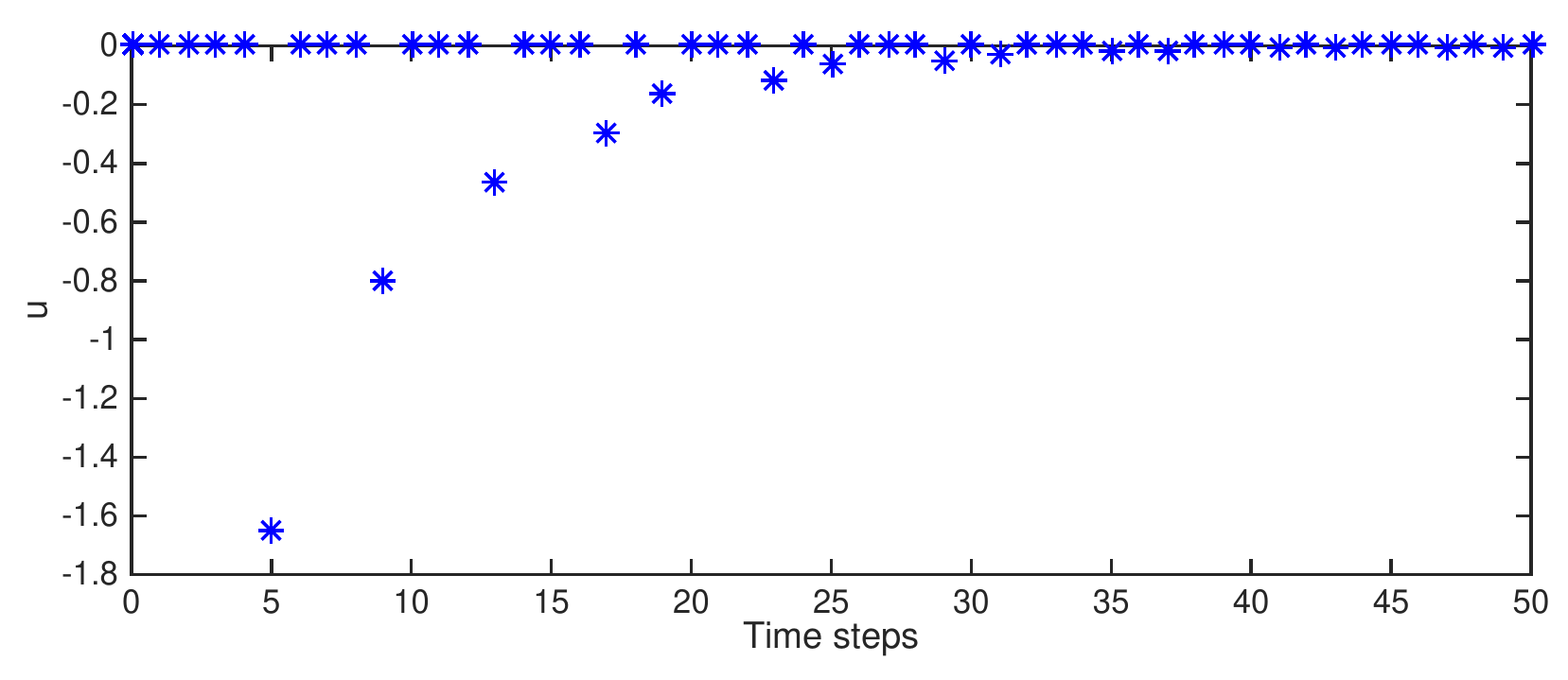}}}
\caption{Simulations of the event-triggered impulsive control inputs for system~\eqref{nsys} with different parameters in the threshold.} \label{fig1u}
\end{figure}

\begin{table}
\tbl{Number of the Event Times on the Time Interval $[0,3000]$.}
{\begin{tabular}{cccc} \toprule
~~~~$a$~~~~ & ~~~~$b$~~~~ &  \multicolumn{2}{c}{~~~Number of event times~~~}  \\ \midrule
$5$ & 0.04 & \hskip9mm 592 & 595 \\
$5$ & 0.07 & \hskip9mm 716 & 719 \\
$24$ & 0.07 & \hskip9mm 713 & 715 \\ \bottomrule
\end{tabular}}
\tabnote{The first column under the category `Number of event times' is for system~\eqref{nsys} with initial condition $x(0)=0.1$, while the second column is derived with initial condition $x(0)=3$.}
\label{table}
\end{table}

\begin{example}\label{example2}
We discuss the linear case of system~\eqref{sys}
\begin{eqnarray}\label{lsys}
\left\{\begin{array}{ll}
x(k+1)=A x(k) + B u(k) \cr
x(0)=x_0
\end{array}\right.
\end{eqnarray}
where $x(k)\in\mathbb{R}^n$, $A\in\mathbb{R}^{n\times n}$, and $B\in\mathbb{R}^{n\times m}$. The control $u$ is in the form of~\eqref{pulse} with $\mathbf{k}(x)=K x$ and $\Gamma\geq 0$, where matrix $K\in\mathbb{R}^{m\times n}$ is the control gain, and event times $\{k_i\}_{i\in\mathbb{N}}$ are to be determined by trigger~\eqref{trigger} with Lyapunov function $V(x)=\|x\|$ and sampling period $\Delta=1$.
\end{example}

It can be seen that (A1) of Assumption~\ref{Assumption1} holds with $\alpha(s)=\beta(s)=s$ for $s\in\mathbb{R}^+$, and (A2) is satisfied with $c=\|A\|$ since
\[
V(f(x,0))=V(Ax)=\|Ax\|\leq \|A\|V(x).
\]
For (A3) of Assumption~\ref{Assumption1}, we have $g(x)=Ax$ and $g^{\Gamma}(x)=A^{\Gamma}x$, then
\begin{align*}
V(f(g^{\Gamma}(x),\mathbf{k}(x))) & = V\left(AA^{\Gamma}x+ BKx\right) \cr
        &\leq \left\|A^{\Gamma+1}+BK\right\| V(x),
\end{align*}
which implies that (A3) is true with $\rho= \|A^{\Gamma+1}+BK \|$. Therefore, we conclude from Theorem~\ref{Th1} that if
\begin{align}\label{lcondition}
\left\|A^{\Gamma}+BK\right\| \leq \eta
\end{align}
holds with $\eta={(1-b)^{\Gamma+2}}/{\|A\|}$, then the closed-loop system is asymptotically stable for any $x_0\in\mathcal{B}(\beta^{-1}(a))$. In the simulation, we consider the following parameters 
\[
A=\begin{bmatrix}
0.1 & 1.2 \\
0.007 & 1.05
\end{bmatrix} \textrm{~~and~~} B=\begin{bmatrix}
300 & 200 \\
0.5 & 0.001
\end{bmatrix}
\]
from \cite{AE-VD-KK:2010} with $a=5$ and $b=0.05$. The control gain $K$ can be readily obtained by solving the linear matrix inequality (LMI) equivalent to~\eqref{lcondition} with the help of the LMI Toolbox in MATLAB, once the delay $\Gamma$ is specified. Nevertheless, in order to demonstrate the delay effects on the dynamics of the closed-loop system, we consider a particular control gain
\[
K=\begin{bmatrix}
0 & -2 \\
0 & ~3
\end{bmatrix}
\]
instead, so that~\eqref{lcondition} is satisfied with $\Gamma=0$, $1$, and $2$. See Figure~\ref{fig2} for numerical simulations. {Intuitively, larger actuation delay $\Gamma$ allows the Lyapunov function to go beyond the triggering threshold further, and then, under the same feedback control law, the event is triggered with a higher frequency. It can be observed in Figure~\ref{fig2} that larger $\Gamma$ leads to more frequent occurrence of the event.}

\begin{figure}
\centering
\subfloat[$\Gamma=0$]{%
\resizebox*{8cm}{!}{\label{fig2.0}\includegraphics{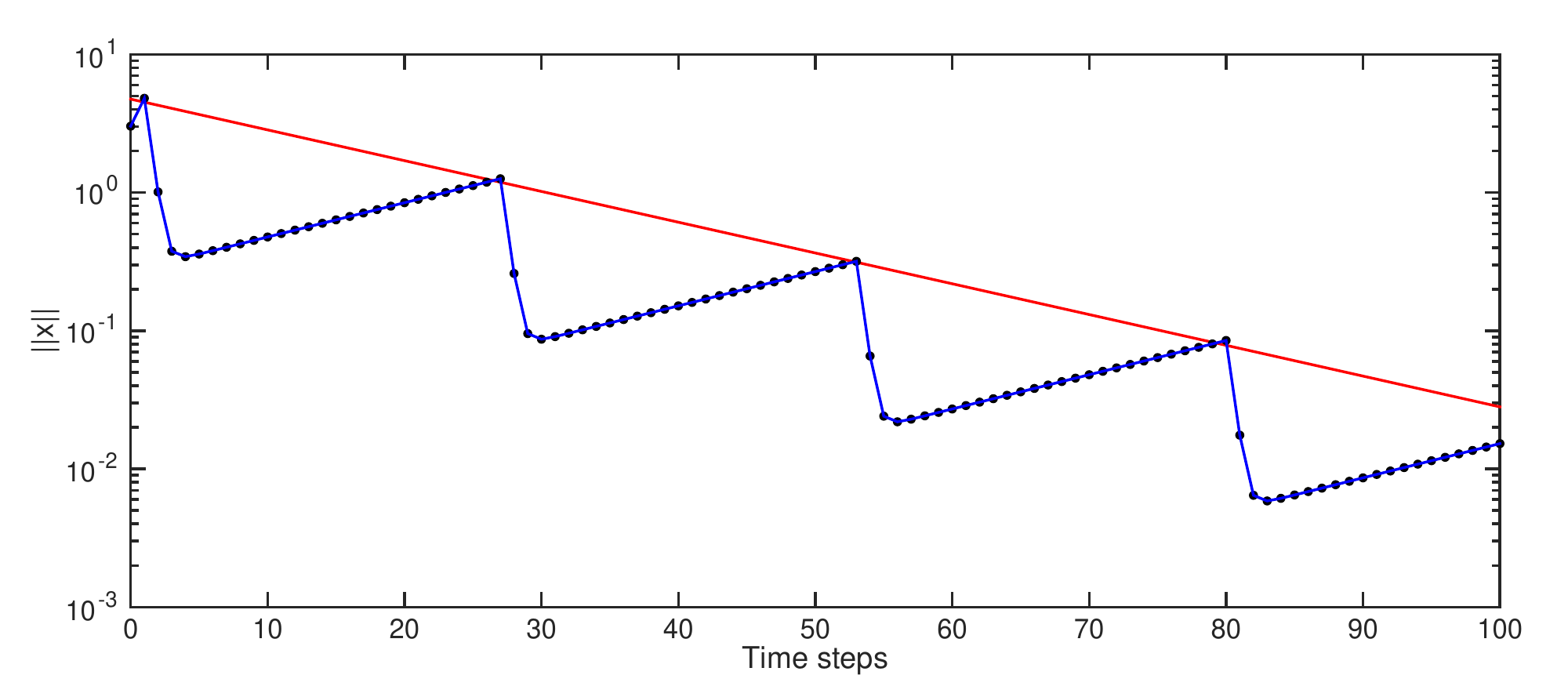}}}\hspace{5pt}
\subfloat[$\Gamma=1$]{%
\resizebox*{8cm}{!}{\label{fig2.1}\includegraphics{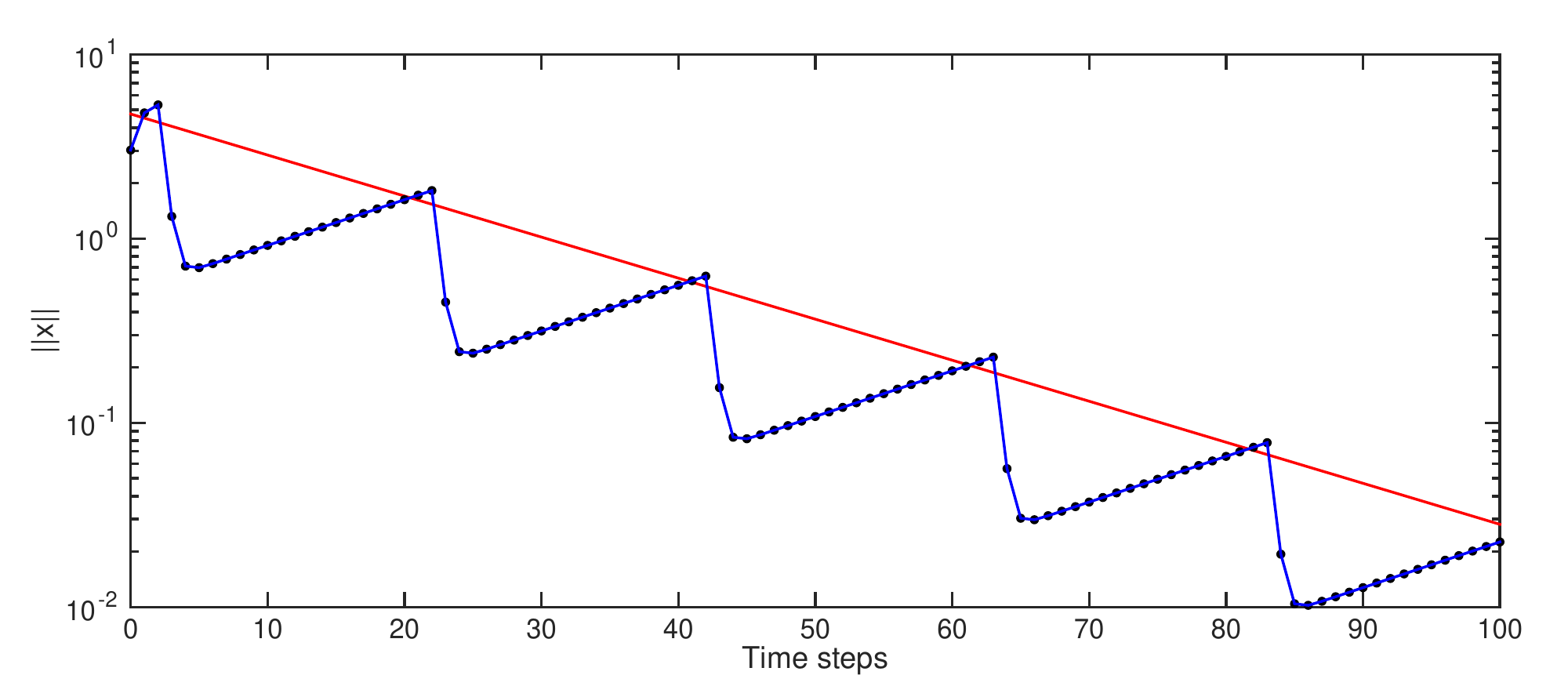}}}\hspace{5pt}
\subfloat[$\Gamma=2$]{%
\resizebox*{8cm}{!}{\label{fig2.2}\includegraphics{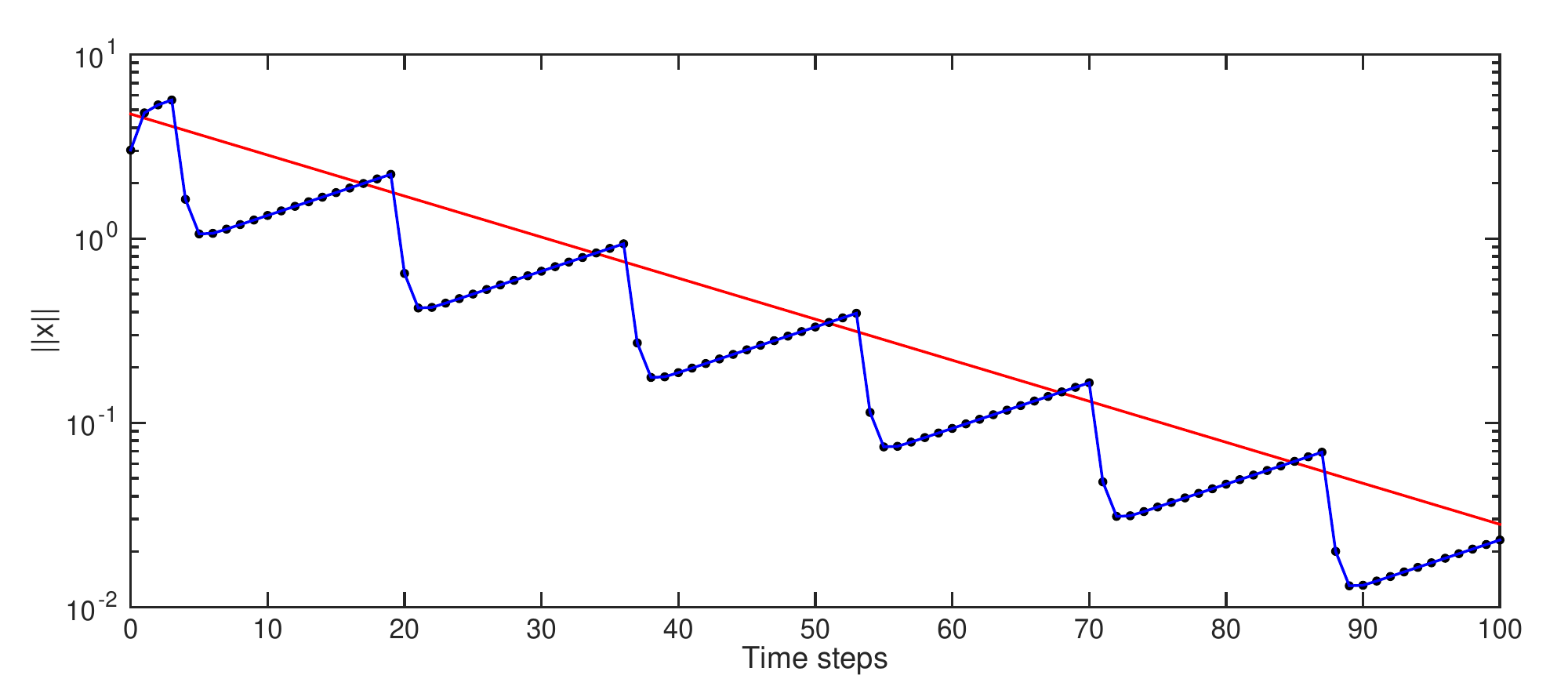}}}
\caption{Simulations of system~\eqref{lsys} with different impulse delays. The red curves represent the threshold lines with $a=5$ and $b=0.05$. The values of $V$ at different time steps are indicated by the black dots, which are traced by the blue curves chronologically. To clearly observe when $V$ passes the threshold line, the $\|x\|$-axes in the above plots are scaled logarithmically.} 
\label{fig2}
\end{figure}

{
\begin{example}\label{example3}
Consider the following nonlinear system
\begin{eqnarray}\label{network}
x(k+1)=C x(k) +A \mathcal{F}(x(k))+ B u(k),
\end{eqnarray}
where $x(k)=[x_1(k),x_2(k),...,x_n(k)]^\top\in\mathbb{R}^n$, $u(k)\in\mathbb{R}^m$, $C=\textrm{diag}(c_1,c_2,...,c_n)$ with constants $0<|c_i|<1$ for $i=1,2,...,n$, matrices $A\in\mathbb{R}^{n\times n}$ and $B\in\mathbb{R}^{n\times m}$, and $\mathcal{F}(x(k))=[f_1(x_1(k)),f_2(x_2(k)),...,f_n(x_n(k)) ]^\top$ with functions $f_i$ satisfying Lipschitz conditions, that is, there exists $l_i>0$ such that $|f_i(y)-f_i(z)|\leq l_i|y-z|$ for any $y,z\in\mathbb{R}^n$ and $i=1,2,...,n$. Denote $L=\textrm{diag}(l_1,l_2,...,l_n)$.
\end{example}

In this example, we consider $\Gamma=0$ and the state feedback control $u(k)=Kx(k)$ where the control gain $K$ is an $m\times n$ matrix. The event times $\{k_i\}_{i\in\mathbb{N}}$ are to be determined by \eqref{trigger} with Lyapunov function $V(x)=x^\top x$.

When $k\not=k_i$, we have
\begin{align*}
         &V(x(k+1))\cr
         &= x^\top(k) C^2 x(k) + x^\top(k)CA\mathcal{F}(x(k)) + F^\top(x(k))A^\top Cx(k) +\mathcal{F}^\top (x(k)) A^\top A\mathcal{F}(x(k))\cr
         &\leq 2x^\top(k) \left( C^2+\|A\|^2 L^2 \right)x(k),
\end{align*}
and then condition (A2) in Assumption~\ref{Assumption1} is satisfied with 
\[c=2\left( \max_{i}\{c_i^2\}  + \|A\|^2\max_{i}\{l^2_i\}\right).\]

Similarly, for $k=k_i+1$ we get 
\begin{equation}\label{ex3-1}
V(x(k_i+1))\leq 2 x^\top(k_i)\left( (C+BK)^\top(C+BK) +\|A\|^2 L^2 \right) x(k_i).
\end{equation}
If there exists a $\rho>0$ such that
\begin{equation}\label{ex3-2}
\begin{bmatrix}
I & C+BK\cr
C^\top + K^\top B^\top & \frac{1}{2}\rho I -\|A\|^2L^2 
\end{bmatrix}\geq 0,
\end{equation}
where $I$ denotes the $n\times n$ identity matrix, then the Schur complement with~\eqref{ex3-1} and~\eqref{ex3-2} implies that condition (A3) of Assumption~\ref{Assumption1} holds. Hence, all the conditions of Assumption~\ref{Assumption1} are satisfied. If we further assume that~\eqref{condition} holds, then the event-triggered control system~\eqref{network} is asymptotically stable.

It should be mentioned that the above stability analysis on event-triggered control system~\eqref{network} has wide applications on stabilization and synchronization problems of discrete-time Hopfield neural networks (see, e,g, \cite{XL-TC:2002,ANM-JAF-HFS:1990}), and the proposed event-triggered impulsive control can dramatically reduce the energy assumption due to the neuron communication, since the lower bound of the inter-event times is not less than $\Gamma+2$.

}

\section{Conclusions}\label{Sec6}
This paper studied delayed impulsive stabilization of discrete-time systems. A new periodic event-triggering scheme with two adjustable parameters was designed to determine the moments of updating the control inputs. Sufficient conditions on the parameters, the sampling period, and time delays were derived to ensure asymptotic stability of the closed-loop systems. Three examples were provided to demonstrate the theoretical result. Along the line of this research, the extension of the proposed event-triggering scheme for discrete-time delay systems deserves future investigation. {Extensions can also be made to singular discrete-time systems (see, e.g., \cite{YH-YK-CG:2017}). The destabilization effects of the impulse delays were analyzed in this study. However, it has been shown that the delays in the impulse can contribute to the stabilization of the impulsive control systems (see, e.g, \cite{XL-SS:2016}). Therefore, the future research can also focus on the positive delay effects on the stabilization of discrete-time systems via event-triggered impulsive control. }

%
%
%

%
\section*{Funding}

This work was supported by the Natural Sciences and Engineering Research Council of Canada (NSERC) under grants RGPIN-2020-03934 and RGPIN-2022-03144.

%
%
%
%
%
%
%

\end{document}